\documentclass[12pt]{amsart}
\usepackage{amsfonts}
\usepackage{bbm}
\usepackage{amsfonts,amssymb,amsmath,amsthm}
\usepackage{url}
\usepackage{enumerate}
\usepackage{bbm}
\usepackage{amsfonts}
\usepackage{mathrsfs}
\usepackage{amsmath,amscd,amsthm,amsfonts}
\usepackage[all]{xy}
\usepackage{color}
\usepackage[latin1]{inputenc}
\usepackage[pdftex, colorlinks, citecolor=red, backref=page]{hyperref}

\newtheorem{thm}{Theorem}[section]
\newtheorem{defi}{Definition}[section]
\newtheorem{lem}[thm]{Lemma}

\newtheorem{prop}[thm]{Proposition}
\newtheorem{rem}{Remark}[section]

\theoremstyle{definition}
\numberwithin{equation}{section}

\author {Wenda Zhang}
\address{College of Mathematics and Statistics, Chongqing Jiaotong University, Chongqing, China \ 400074}
\email{wendazhang951@aliyun.com}
\author{zhiqiang li}
\address{College of Mathematics and Statistics, Chongqing University, Chongqing, China \ 401331}
\email{zqli@cqu.edu.cn}
\author{yunhua zhou}
\address{College of Mathematics and Statistics, Chongqing University, Chongqing, China \ 401331}
\email{zhouyh@cqu.edu.cn}

\begin{document}

\title[sub-additive unstable topological pressure]{unstable topological pressure of partially hyperbolic diffeomorphisms with sub-additive potentials}

\begin{abstract}
In this paper, we introduce the unstable topological pressure for $C^1$-smooth partially hyperbolic diffeomorphisms with sub-additive
potentials. Moreover, without any additional assumption, we have established the expected variational principle which connects this unstable topological pressure and the unstable measure theoretic entropy, as well as the corresponding \emph{Lyapunov exponent.}
\end{abstract}

\maketitle

\section{Introduction}
It is well known that the topological pressure for additive potentials was first introduced by Ruelle
for expansive maps, see \cite{Ruelle}. In \cite{Pesin1}, Pesin and Pitskel defined topological
pressures for non-compact sets, which is a generalization of Bowen's topological entropy on
non-compact sets in \cite{Bowen1}, and they proved a variational principle under some supplementary conditions. On the other hand, in \cite{K. Falconer} and \cite{L. Barreira1}, Falconer and Barreira investigated topological pressures for non-additive (including sub-additive ones) potentials and obtained variational principles under some restrictions. In \cite{Huang1}, Cao, Feng, and Huang established the variational principle of topological pressure for sub-additive potentials without any additional assumptions. 
Topological pressures, variational principles, and equilibrium states play fundamental
roles in statistical mechanics, ergodic theory, and dynamical systems, see the books \cite{Bowen,Walters2}.

In the category of differentiable dynamics, dynamical invariants such as entropy and pressure are developed for diffeomorphisms on closed Riemannian manifolds, especially for $C^1$-smooth partially hyperbolic diffeomorphisms. A fundamental result is due to Hu, Hua, and Wu (\cite{Hu1}), they introduced the so called unstable topological and unstable metric entropy, obtained the corresponding Shannon-McMillan-Breiman theorem and variational principle. Along this line, Hu, Wu, and Zhu investigated the unstable topological pressure for additive potentials in \cite{Zhu2}.

In this paper, we study the unstable topological pressure in a more general setting, i.e., for $C^1$-smooth partially hyperbolic diffeomorphisms with sub-additive potentials, and set up the expected variational principle. Although the project originates from purely topological setting, we must deal with measure disintegration in differentiable dynamics with respect to unstable manifolds, which causes some difficulties and complexity. Moreover, we have to handle the \emph{Lyapunov exponent} of sub-additive potentials entirely rather than focus only on the continuous potential in additive case.  

\begin{thm}\label{variational principle}
Let $f : M \rightarrow M$ be a $C^{1}$-smooth partially hyperbolic diffeomorphism and $\mathcal{G} = \{\log g_{n}\}
_{n=1}^{\infty}$ be a sequence of sub-additive potentials of $f$ on $M$. Then$$
P^{u}(f,\mathcal{G})=\sup\{h^{u}_{\mu}(f)+\mathcal{G}_{*}(\mu)\mid\mu\in \mathcal{M}_{f}(M)\}.
$$
Moreover,
$$
P^{u}(f,\mathcal{G})= \sup\{h^{u}_{\mu}(f)+\mathcal{G}_{*}(\mu)\mid\mu\in \mathcal{M}^{e}_{f}(M)\}.
$$
\end{thm}
(All involved terms and notation are defined in Section 2.)

The paper is organized as follows. In Section 2, we set up notation and give definitions of unstable topological pressure for sub-additive potentials. Section 3 consists of the proof of Theorem \ref{variational principle}.

\section{Notation and definitions}
Throughout the paper, we focus on the dynamical system $(M, f)$, where $M$ is a finite dimensional, smooth, connected, and compact Riemannian manifold
without boundary; and $f : M \rightarrow M$ is a $C^{1}$-smooth partially hyperbolic diffeomorphism. 
We say $f$ is
\emph{partially hyperbolic}, if there exists a nontrivial $Df$-invariant
splitting $TM=E^{s}\bigoplus E^{c}\bigoplus E^{u}$ of the tangent bundle into stable, central, and unstable
distributions, such that all unit vectors $v^{\sigma}\in E^{\sigma}_{x} (\sigma=s, c, u)$ with $x \in M$ satisfy
$$\parallel D_{x}fv^{s} \parallel< \parallel D_{x}fv^{c} \parallel<\parallel D_{x}fv^{u} \parallel,$$
and
$$\parallel D_{x}f\mid_{E^{s}_{x}}\parallel< 1 \;\;\mbox{and} \;\;\parallel D_{x}f^{-1}\mid_{E^{u}_{x}}\parallel< 1,$$
for some suitable Riemannian metric on $M$. The stable distribution $E^{s}$ and unstable
distribution $E^{u}$ are integrable to the stable and unstable foliations $W^{s}$ and $W^{u}$
respectively such that $TW^{s} = E^{s}$ and $TW^{u} = E^{u}$.

We denote by $\mathcal{M}(M)$ the set of Borel probability measures on $M$, by $\mathcal{M}_f(M)$ the subset of invariant ones of $f$ on $M$, and by $\mathcal{M}^e_f(M)$ the subset of ergodic ones.

Let us gather some necessary preliminaries for the unstable metric entropy in \cite{Hu1}.

Take $\epsilon_{0} > 0$
small. Let $\mathcal{P} = \mathcal{P}_{\epsilon_{0}}$ denote the set of finite Borel partitions $\alpha$ of $M$ whose elements
have diameters smaller than or equal to $\epsilon_{0}$, that is, $\mbox{diam}\;\alpha := \sup\{\mbox{diam}\;A : A \in
\alpha\} \leq \epsilon_{0}$. For any partition $\xi$ of $M$ and any $x\in M$, we denote by $\xi(x)$ the element of $\xi$ containing $x$.
For each $\beta \in \mathcal{P}$ we can define a finer partition $\eta$  such that $\eta(x) =\beta(x) \cap W^{u}_{loc}(x)$ for each $x \in M$, where $W^{u}_{loc}(x)$ denotes the local unstable manifold
at $x$ whose size is greater than the diameter $\epsilon_{0}$ of $\beta$. Since $W^{u}$ is a continuous
foliation, $\eta$ is a measurable partition with respect to any Borel probability measure
on $M$.

Let $\mathcal{P}^{u}$ denote the set of partitions $\eta$ obtained in this way and \emph{subordinate
to unstable manifolds}. Here a partition $\eta$ of $M$ is said to be subordinate to unstable
manifolds of $f$ with respect to a measure $\mu$ if for $\mu$-almost every $x, \eta(x) \subseteq W^{u}(x)$
and contains an open neighborhood of $x$ in $W^{u}(x)$. It is clear that if $\alpha \in \mathcal{P}$
satisfies $\mu(\partial\alpha) = 0$, where $\partial\alpha := \cup_{A\in \alpha}\partial A$, then the corresponding $\eta$ given by
$\eta(x) = \alpha(x) \cap W^{u}_{loc}(x)$ is a partition subordinate to unstable manifolds of $f$.

Given any probability measure $\nu$ and any measurable partition $\eta$ of $M$, a classical result of Rokhlin (cf.\,\cite{Rohlin}) says that there exists a system of conditional measures with respect to $\eta$. which is a family of probability measures $\{\nu^{\eta}_{x} : x \in M\}$ with $\nu^{\eta}_{x}(\eta(x))= 1$
such that for every measurable set $B \subseteq M, x \mapsto \nu^{\eta}_{x}(B)$ is measurable and
$$\nu(B) =\int_{X}\nu^{\eta}_{x}(B)d\nu(x).$$ This is called the \emph{canonical system of conditional measures}
of $\nu$ over $\eta$ or the measure disintegration of $\nu$ over $\eta$. It is essentially
unique in the sense that two such systems coincide with respect to a set of points with full $\nu$-measure.

In  \cite{Hu1}, the authors defined the unstable metric entropy as follows.
\begin{defi}
For any measurable partitions $\alpha\in \mathcal{P}$ and $\eta\in\mathcal{P}^{u}$, set
$$H_{\mu}(\alpha|\eta):=-\int_{M}\log\mu^{\eta}_{x}(\alpha(x))d\mu(x),$$
 the unstable metric entropy of $f$  is defined as
$$h^{u}_{\mu}(f)=\sup\limits_{\eta\in\mathcal{P}^{u}}h_{\mu}(f|\eta),$$
where
$$h_{\mu}(f, \alpha|\eta)=\limsup_{n\rightarrow\infty}\frac{1}{n}H_{\mu}(\alpha_{0}^{n-1}|\eta)\;\text{and}\;h_{\mu}(f|\eta)=\sup\limits_{\alpha\in\mathcal{P}}h_{\mu}(f, \alpha|\eta)$$
\end{defi}

\begin{rem}
In \cite{Hu1}, the authors proved that $h_{\mu}(f|\eta)$ is independent of $\eta\in\mathcal{P}^u$. Moreover, for any ergodic measure $\mu$, any $\alpha\in\mathcal{P}_{\epsilon}$($\epsilon$ small enough), $\eta\in \mathcal{P}^u$,  one has $h^u_{\mu}(f)=h_{\mu}(f|\eta)=h_{\mu}(f, \alpha|\eta)$. Moreover, the limit sup above actually means the limit (See Lemma 2.8, Theorem A, and Corollary A.2 there). From now on we shall use this fact frequently without mentioning the reference any more.
\end{rem}

\begin{defi}
Given a sequence of continuous functions $\mathcal{G} = \{\log g_{n}\}
_{n=1}^{\infty}$ on $M$, $\mathcal{G}$ is called a sequence of sub-additive potentials of $f$ if
\begin{equation}\label{sub-additive}
\log g_{m+n}(x)\leq \log g_{n}(x)+\log g_{m}(f^{n}x), \forall x\in M, m,n\in\mathbb{N}.
\end{equation}
\end{defi}

\begin{rem}
For any $\mu\in\mathcal{M}_f(M)$, set
$$\mathcal{G}_{*}(\mu)=\lim\limits_{n\rightarrow \infty}\frac{1}{n}\int\log g_{n}d\mu,$$
and $\mathcal{G}_{*}(\mu)$ is called the \emph{Lyapunov exponent} of $\mathcal{G}$ with respect to $\mu$. It takes values in
$[-\infty, +\infty)$. Moreover, the Sub-additive Ergodic Theorem (cf. \cite{Walters2}, Theorem 10.1) implies that for an ergodic measure $\mu$, one has
$$\mathcal{G}_{*}(\mu)=\lim\limits_{n\rightarrow \infty}\frac{1}{n}\log g_{n} \;a.e.\;x\in M.$$
\end{rem}

Denote by $d^{u}$ the metric induced by the Riemannian
structure on the unstable manifold and let $d^{u}
_{n}(x, y) = \max\limits_{0\leq j\leq n-1} d^{u}(f^{j}(x), f^{j}(y))$.
Denote by $W^{u}(x, \delta)$ the open ball inside $W^{u}(x)$ with center $x$ and radius $\delta$ with
respect to $d^{u}$. Let $E$ be a subset of points in $\overline{W^{u}(x, \delta)}$ with pairwise $d^{u}_{n}$-distances at least $\epsilon$, such a set is called an $(n, \epsilon) \,u$-separated subset of $\overline{W^{u}(x, \delta)}$. Note that $M$ is compact and $W^{u}$ is a continuous foliation, then for any $\delta>0$ small enough, there exists a $C > 1$ such that for any $x\in M$, one has
\begin{equation}\label{distance}
 d(y, z)\leq d^{u}(y, z)\leq C d(y, z), \text{for any} \;y, z\in\overline{ W^{u}(x,\delta)}.
\end{equation}

\begin{defi}\label{sep-pressure}
Let $\mathcal{G} = \{\log g_{n}\}
_{n=1}^{\infty}$  be a sequence of  sub-additive potentials of $f$.
Set
$$\begin{aligned}
&P_{n}^{u}(f, \mathcal{G}, \epsilon, x, \delta)\\
:=&\sup\left\{\sum_{y\in E}g_{n}(y)\mid E\,\text{ is an}\, (n, \epsilon)\, u\text{-separated subset of}\; \overline{W^{u}(x, \delta)}\right\}
\end{aligned},$$
and
$$P^{u} (f, \mathcal{G}, \epsilon, x,\delta):=\limsup_{n\rightarrow\infty}\frac{1}{n}\log P_{n}^{u}(f, \mathcal{G}, \epsilon, x, \delta).$$
Define $$P^{u}(f, \mathcal{G}, x, \delta):=\lim\limits_{\epsilon\rightarrow 0}P^{u} (f, \mathcal{G}, \epsilon, x,\delta).$$
The unstable topological pressure of $f$ with respect to $\mathcal{G}$ is defined as
$$P^{u}(f, \mathcal{G}):=\lim\limits_{\delta\rightarrow 0}\sup\limits_{x\in M}P^{u}(f, \mathcal{G}, x, \delta).$$
\end{defi}

We can also define unstable
topological pressure using $(n,\epsilon)\; u$-spanning sets as follows.

A subset $F \subset W^{u}(x)$ is called an $(n,\epsilon)\; u$-spanning set of $\overline{W^{u}(x, \delta)}$ if $\overline{W^{u}(x, \delta)}\subset \cup_{y\in F}B_{n}^{u}(y,\epsilon)$, where $B^{u}
_{n}(y,\epsilon) = \{z \in W^{u}(x) : d^{u}
_{n}(y, z) \leq\epsilon\}$ is the $(n,\epsilon)\; u$-Bowen
ball around $y$.

\begin{defi}\label{span-pressure}
Let $\mathcal{G} = \{\log g_{n}\}
_{n=1}^{\infty}$  be a sequence of  sub-additive potentials of $f$.
Set
$$\begin{aligned}
&Q_{n}^{u}(f, \mathcal{G}, \epsilon, x, \delta)\\
:=&\inf\left\{\sum_{y\in F}\sup_{z\in B_{n}^{u}(y,\epsilon)}g_{n}(z)\mid F\;\text{is an} \;(n, \epsilon)\, u\text {-spanning subset of } \overline{W^{u}(x, \delta)}\right\},
\end{aligned}$$

$$Q^{u}(f, \mathcal{G}, \epsilon, x, \delta):=\limsup\limits_{n\rightarrow\infty}\frac{1}{n}\log Q_{n}^{u}(f, \mathcal{G}, \epsilon, x, \delta)$$
and
$$Q^{u}(f, \mathcal{G},  x, \delta ):=\lim\limits_{\epsilon\rightarrow 0}Q^{u}(f, \mathcal{G}, \epsilon, x, \delta).$$
Define
$$P^{u\ast}(f, \mathcal{G}):=\lim\limits_{\delta\rightarrow 0}\sup\limits_{x\in M}Q^{u}(f, \mathcal{G}, x, \delta).$$
\end{defi}

The two definitions above actually coincide.
\begin{prop}
Suppose $\mathcal{G}=\{\log g_{n}\}^{\infty}_{n=1}$ is a sequence of sub-additive potentials of $f$. Then
for any $x\in M$ and $\delta>0$, one has
$Q^{u}(f, \mathcal{G}, x, \delta)=P^{u}(f, \mathcal{G}, x, \delta),$ and hence
$P^{u}(f, \mathcal{G})=P^{u\ast}(f, \mathcal{G}).$
\end{prop}

\begin{proof} First we show for any $x\in M$ and $\delta>0$,
\begin{equation}\label{sep-span 1}
Q^{u}(f, \mathcal{G}, \ x, \delta)\geq P^{u}(f, \mathcal{G}, \ x, \delta).
\end{equation}
For any $\epsilon>0$, suppose $F$ is an arbitrary $(n,\epsilon/2)\,u$-spanning subset  and  $E$ is an arbitrary $(n,\epsilon)\,u$-separated
subset of $\overline{W^u(x, \delta)}$, respectively. For each $y\in E$, one can choose some $z\in F$ with $d^{u}_{n}(y, z)\leq\dfrac{\epsilon}{2}$, and then define a map $\varphi: \;E\rightarrow\,F$ by $\varphi(y)=z$. Then $\varphi$ is injective.  Therefore,
$$ \sum_{\varphi(y)\in F}\sup_{z\in B_{n}^{u}(\varphi(y),\epsilon/2)}g_{n}(z)\geq\sum_{y\in E}g_{n}(y),$$
and then
$$Q_{n}^{u}(f, \mathcal{G}, \epsilon/2, x, \delta)\geq P_{n}^{u}(f, \mathcal{G}, \epsilon, x, \delta).$$
This immediately yields (\ref{sep-span 1}).

On the other hand, for any given $n\in \mathbb{N}$ and $\epsilon>0$, one can choose
$y_{1}\in \overline{W^{u}(x, \delta)}$ with $$g_{n}(y_1)=\sup\limits_{y\in \overline{W^{u}(x, \delta)}}g_{n}(y),$$ and then choose
$ y_{2}\in\overline{W^{u}(x, \delta)}\setminus B^{u}_{n}(y_{1},\epsilon)$ with
$$g_{n}(y_2)=\sup\limits_{y\in \overline{W^{u}(x, \delta)}\setminus B^{u}_{n}(y_{1},\epsilon)}g_{n}(y).$$ One can continue this process. More precisely, in step $m$ we choose $ y_{m}\in\overline{W^{u}(x, \delta)}\setminus \bigcup^{m-1}_{j=1}B^{u}_{n}(y_{j},\epsilon)$ with $$g_{n}(y_m)=\sup\limits_{y\in \overline{W^{u}(x, \delta)}\setminus \bigcup^{m-1}_{j=1}B^{u}_{n}(y_{j},\epsilon)}g_{n}(y).$$ Suppose this process stops at certain step $l$, and produces a maximal $(n,\epsilon)\,u$-separated set $E=\{y_{1},\cdots,y_{l}\}$, which implies that $E$ is also an $(n,\epsilon)\,u$-spanning set of $\overline{W^u(x, \delta)}$. Therefore,
$$\begin{aligned}
&Q_{n}^{u}(f, \mathcal{G}, \epsilon, x, \delta)\\
\leq& \sum_{y\in E}\exp\left(\sup\limits_{z\in B^{u}_{n}(y,\epsilon)}\log g_{n}(y)\right)=\sum_{y\in E}g_{n}(y)\\
\leq& \sup\left\{\sum_{y\in F}g_{n}(y)\mid F\;\text{ is an}\, (n, \epsilon)\, u\text{-separated subset of } \overline{W^{u}(x, \delta)}\right\}\\
=&P_{n}^{u}(f, \mathcal{G}, \epsilon, x, \delta).
\end{aligned}$$
and so
\begin{equation}\label{span-sep}
Q^{u}(f, \mathcal{G}, \ x, \delta)\leq P^{u}(f, \mathcal{G}, \ x, \delta).
\end{equation}

Combining (\ref{sep-span 1}) with (\ref{span-sep}), one has
$Q^{u}(f, \mathcal{G}, x, \delta)=P^{u}(f, \mathcal{G}, x, \delta),$ and hence
$P^{u}(f, \mathcal{G})=P^{u\ast}(f, \mathcal{G}).$
\end{proof}

\begin{rem}
With a quite similar argument as the proof of Lemma 2.1 in \cite{Zhu2}, we can show that the unstable topological pressure for sub-additive potentials we defined above is independent of the choice of $\delta$.
\end{rem}

Another formulation of the unstable topological
pressure is using open covers. Let $Cov^{O}_{(x,\delta)}$ be the collection of all finite open covers of $\overline{W^{u}(x, \delta)}$ and $Cov^{B}_{(x,\delta)}$ be the collection of all finite Borel covers of $\overline{W^{u}(x, \delta)}$, i.e., each member of the cover is a Borel subset. Suppose that $\mathcal{V}$ and $\mathcal{U}$ are two covers of $\overline{W^{u}(x, \delta)}$. Denote by $\mathcal{V}\succeq \mathcal{U}$ if the cover $\mathcal{V}$ is a refinement of the cover $\mathcal{U}$. Set $\text{diam}(\mathcal{U}):=\max\{\text{diam}(A)|\, A\in\mathcal{U}\}$.

Given $\mathcal{U} \in Cov^{O}_{(x,\delta)}$, set $\mathcal{U}^{n}_{m}:=
\bigvee^{n}_{i=m} f^{-i}\mathcal{U},$ and define
$$p_{n}^{u} (f, \mathcal{G},\mathcal{U},  x, \delta):=\inf\left\{\sum_{B\in \mathcal{V}}\sup\limits_{y\in B\cap \overline{W^{u}(x, \delta)}}g_{n}(y)\mid\mathcal{V}\in Cov^{B}_{(x,\delta)}, \mathcal{V}\succeq\mathcal{U}^{n-1}_{0}\right\},$$
where the infimum is taken over all Borel covers refining $\mathcal{U}$.

\begin{rem}\label{ocover-pressure}
For any $x\in M$, any $\delta>0$, and any  $\mathcal{U} \in Cov^{O}_{(x,\delta)}$,
 the sequence $\{\log p_{n}^{u} (f, \mathcal{G}, \mathcal{U}, x, \delta)\}_{n\geq 1}$ is sub-additive.
\end{rem}
 In fact, for any positive integers $n$ and $m$, observe that if $\alpha\in Cov^{B}_{(x,\delta)}$ such that $\alpha\succeq \mathcal{U}^{n-1}_{0}$, and $\beta\in Cov^{B}_{(x,\delta)}$ with $ \beta\succeq\mathcal{U}^{m-1}_{0}$, then $\alpha\vee f^{-n}\beta\succeq \mathcal{U}^{n+m-1}_{0}$ and
$$\sum_{A\in \alpha\vee f^{-n}\beta}\sup\limits_{y\in A\cap \overline{W^{u}(x, \delta)}}g_{n+m}(y)\leq \left(\sum_{B\in \alpha}\sup\limits_{y\in B\cap \overline{W^{u}(x, \delta)}}g_{n}(y)\right)\cdot\left(\sum_{C\in \beta}\sup\limits_{z\in C\cap \overline{W^{u}(x, \delta)}}g_{m}(z)\right).$$
Therefore,
$$\log p_{n+m}^{u} (f, \mathcal{G}, \mathcal{U}, x, \delta)\leq \log p_{n}^{u} (f, \mathcal{G}, \mathcal{U}, x, \delta)+\log p_{m}^{u} (f, \mathcal{G}, \mathcal{U}, x, \delta).$$

\begin{defi}\label{ocover-pressure}
Set $$\widetilde{P}^{u}(f, \mathcal{G}, \mathcal{U}, x, \delta):=\lim\limits_{n\rightarrow\infty}\frac{1}{n}\log p_{n}^{u} (f, \mathcal{G}, \mathcal{U}, x, \delta),$$
and
 $$\widetilde{P}^{u}(f, \mathcal{G},x, \delta):=\sup\limits_{\mathcal{U} \in Cov^{O}_{(x,\delta)}}\widetilde{P}^{u}(f, \mathcal{G}, \mathcal{U}, x, \delta).$$
We define $$\widetilde{P}^{u}(f, \mathcal{G}):=\lim\limits_{\delta\rightarrow0}\sup\limits_{x\in M}\widetilde{P}^{u}(f, \mathcal{G},x, \delta).$$
\end{defi}

\begin{rem}
We claim that
\begin{equation}\label{cover-limit}
\widetilde{P}^{u}(f, \mathcal{G}, x, \delta)=\lim\limits_{diam(\mathcal{U})\rightarrow 0}\widetilde{P}^{u}(f, \mathcal{G}, \mathcal{U}, x, \delta).
\end{equation}
Indeed if $\mathcal{V}, \mathcal{U}\in Cov^{O}_{(x,\delta)}$ and $diam(\mathcal{V})$ is less than the Lebesgue number of $\mathcal{U}$, then by the definition one gets
$$p_{n}^{u} (f, \mathcal{G},\mathcal{V},  x, \delta)\geq p_{n}^{u} (f, \mathcal{G},\mathcal{U},  x, \delta).$$
\end{rem}

Next we show that this formulation in terms of open covers is equivalent to the previous definitions.

\begin{prop}Suppose $\mathcal{G}=\{\log g_{n}\}^{\infty}_{n=1}$ is a sequence of sub-additive potentials of $f$. Then
for any $x\in M$ and $\delta>0$, one has
$$\widetilde{P}^{u}(f, \mathcal{G}, x, \delta)=P^{u}(f, \mathcal{G}, x, \delta)=Q^{u}(f, \mathcal{G}, x, \delta),$$ and hence
$$\widetilde{P}^{u}(f, \mathcal{G})=P^{u}(f, \mathcal{G})=P^{u\ast}(f, \mathcal{G}).$$
\end{prop}
\begin{proof}
First we show  $$\widetilde{P}^{u}(f, \mathcal{G}, x, \delta)\leq P^{u}(f, \mathcal{G}, x, \delta).$$
For any $\epsilon>0$, any $n\in \mathbb{N}$, any $x\in M$, and any $\delta>0$, we construct an $(n, \epsilon)$ $u$-separated set of $\overline{ W^{u}(x,\delta)}$ in the following way. 
Take $y_{1}\in \overline{ W^{u}(x,\delta)}$ such that $g_{n}(y_{1})=\sup\limits_{y\in \overline{ W^{u}(x,\delta)}}g_{n}(y)$, and then take $y_{2}\in \overline{ W^{u}(x,\delta)}\setminus B^{u}_{n}(y_{1},\epsilon)$ with $$g_{n}(y_{2})=\sup\limits_{y\in \overline{ W^{u}(x,\delta)}\setminus B^{u}_{n}(y_{1},\epsilon)}g_{n}(y).$$
Continue this procedure, in step $l$, choose $y_{l}\in \overline{ W^{u}(x,\delta)}\setminus \bigcup\limits^{l-1}_{i=1}B^{u}_{n}(y_{i},\epsilon)$ such that $$g_{n}(y_{l})=\sup\limits_{y\in \overline{ W^{u}(x,\delta)}\setminus \bigcup\limits^{l-1}_{i=1}B^{u}_{n}(y_{i},\epsilon)}g_{n}(y).$$
By compactness of $\overline{ W^{u}(x,\delta)}$, this process will stop at certain step $N$. Denote by $E=\{y_{1}, \cdots, y_{N}\}$ the set we obtain in this way. Then $E$ is an $(n, \epsilon)$ $u$-separated subset of $\overline{ W^{u}(x,\delta)}$.
Set $\mathcal{U}=\{B^{u}(f^{i}y_{j}, \epsilon)| \;0\leq i\leq n-1,\;1\leq j\leq N\}$.
It is obvious that $\mathcal{U}$ forms an open cover of $\overline{ W^{u}(x,\delta)}$ with $diam(\mathcal{U})\leq 2\epsilon$ in the $d^{u}$ metric.
Set $$\alpha=\{B^{u}_{n}(y_{1},\epsilon), B^{u}_{n}(y_{2},\epsilon)\setminus B^{u}_{n}(y_{1},\epsilon),\cdots, B^{u}_{n}(y_{N},\epsilon)\setminus\bigcup\limits^{N-1}_{i=1}B^{u}_{n}(y_{i},\epsilon)\}.$$
Then $\alpha$ is a Borel cover of $\overline{ W^{u}(x,\delta)}$ and $\alpha\succeq \mathcal{U}^{n-1}_{0}$.
 Then one has
$$\sum\limits_{A\in\alpha}\sup\limits_{y\in A\cap\overline{W^{u}(x,\delta)}}g_{n}(y)=\sum\limits^{N}_{i=1}g_{n}(y_{i})\leq P_{n}^{u}(f, \mathcal{G}, \epsilon, x, \delta).$$
Hence
$$\widetilde{P}^{u}(f, \mathcal{G},\mathcal{U}, x, \delta)\leq P_{n}^{u}(f, \mathcal{G}, \epsilon, x, \delta),$$
Let $\epsilon\rightarrow0$, by (\ref{cover-limit}) one has
$$\widetilde{P}^{u}(f, \mathcal{G},x, \delta)\leq P_{n}^{u}(f, \mathcal{G}, x, \delta).$$

Next we show the inverse inequality $$\widetilde{P}^{u}(f, \mathcal{G}, x, \delta)\geq P_{n}^{u}(f, \mathcal{G}, x, \delta).$$
Suppose $\mathcal{U}$ is an open cover of $\overline{W^{u}(x, \delta)}$ with $diam(\mathcal{U})\leq \epsilon$ in the $d^{u}$ metric. Given any Borel cover $\alpha$ of $\overline{W^{u}(x, \delta)}$ with $\alpha\succeq\mathcal{U}^{n-1}_{0}$, and any $(n, \epsilon)\;u$-separated subset $E$ of $\overline{W^{u}(x, \delta)}$, then any $y\in E$ is contained exactly in one member $A\cap \overline{W^{u}(x, \delta)}$ for some $A\in \alpha$.
So we get
$$\widetilde{P}^{u}(f, \mathcal{G},\mathcal{U}, x, \delta)\geq P_{n}^{u}(f, \mathcal{G}, \epsilon, x, \delta).$$
Combining with (\ref{sep-span 1}), we get the desired conclusion.
\end{proof}

Next we gather some basic properties of unstable topological pressure for sub-additive potentials. They follow in quite a similar way as in the classical case, see Theorem 9.7 in \cite{Walters2}, so we omit the proof to avoid redundancy.  
\begin{prop}\label{propeties}
Let $f : M \rightarrow M$ be a $C^{1}$-smooth partially hyperbolic diffeomorphism. Let $\mathcal{G} = \{\log g_{n}\}
_{n=1}^{\infty}$ and $\mathcal{H} = \{\log h_{n}\}_{n=1}^{\infty}$ be two sequences of sub-additive potentials of $f$ on $M$. Then the following
statements are true.
\begin{enumerate}
  \item $P^{u}(f,c+\mathcal{G})=P^{u}(f,\mathcal{G})+c, \;\text{where} \;c+\mathcal{G}=\{c+\log g_{n}\}_{n\in \mathbb{Z_{+}}}.$
  \item If $\mathcal{G}\leq \mathcal{H},\;i.e.\; \text{for any}\;n\geq 1,\;x\in M,\; f_{n}(x)\leq g_{n}(x),\;\text{then}$
  $$P^{u}(f,\mathcal{G})\leq P^{u}(f,\mathcal{H}),$$
  \item $P^{u}(f,\cdot)\;\text{is convex, that is to say},$
  $$P^{u}(f,p\mathcal{G}+(1-p)\mathcal{H})\leq pP^{u}(f,\mathcal{G})+(1-p)P^{u}(f,\mathcal{H}),$$
  $\text{where}\;p\mathcal{G}+(1-p)\mathcal{H}=\{p\log g_{n}+(1-p)\log h_{n}\}_{n\in \mathbb{Z_{+}}}.$
  \item $$P^{u}(f,\mathcal{G}+\mathcal{H})\leq P^{u}(f,\mathcal{G})+P^{u}(f,\mathcal{H}),$$ where $\mathcal{G}+\mathcal{H}=\{\log g_{n}+\log h_{n}\}_{n\in \mathbb{Z_{+}}}$
  \item $$P^{u}(f,c\mathcal{G})\leq cP^{u}(f,\mathcal{G}),\;\text{if}\;c\geq1,$$ and $$\;P^{u}(f,c\mathcal{G})\geq cP^{u}(f,\mathcal{G}),\;\text{if}\;c\leq1.$$
  \item If we further assume $\mathcal{H}$ is additive, then
  $$P^{u}(f,\mathcal{G}+\mathcal{H}\circ f-\mathcal{H})=P^{u}(f,\mathcal{G}),$$ where $\mathcal{G}+\mathcal{H}\circ f-\mathcal{H}=\{\log g_{n}+\log h_{n}\circ f-\log h_{n}\}_{n\in \mathbb{Z_{+}}}.$
  \item For any $k\in \mathbb{N}$, we have
$$P^{u}(f^{k},\mathcal{G}^{(k)})=kP^{u}(f,\mathcal{G})$$
where $f^{k} := \underbrace{\{f\circ\cdots\circ f\}}_{k \;\text{times}}$ and $\mathcal{G}^{(k)}:=\{\log g_{kn}\}_{n\in \mathbb{Z_{+}}}.$
\end{enumerate}
\end{prop}

\begin{rem}
It is not clear whether the result of $(6)$ is still true if $\mathcal{H}$ is only  sub-additive.
\end{rem}

\section{The proof of the variational principle}
The following result is well known, see e.g. Lemma 9.9 of \cite{Walters2}.
\begin{lem}\label{lemma 2.9}
Given $0\leq  p_{1},\cdots,p_{m}\leq 1$ with $\sum^{m}_{i=1}p_{i}=1$, and
$a_{1},\cdots ,a_{m}\in\mathbb{R}$, then
$$\sum^{m}_{i=1}p_{i}(a_{i}-\log p_{i})\leq \log\sum^{m}_{i=1}e^{a_{i}},$$
and the equality holds if and only if  $p_i=\dfrac{e^{a_{i}}}{\sum^{m}_{i=1}e^{a_{i}}}$.
\end{lem}

The following power rule for unstable entropy is also straightforward.
\begin{lem}\label{lemma 2.10} For any $\mu\in\mathcal{M}_{f}(M)$, one has $h_{\mu}(f^k | \eta)=kh_{\mu}(f | \eta)$, and hence
$h^{u}_{\mu}(f^{k})=kh^{u}_{\mu}(f).$
\end{lem}

\textbf{Now we proceed to prove Theorem \ref{variational principle}.}
\begin{proof}
 We divide the proof into three steps.

$\mathit{Step\, 1}.$
In this step will show that for any $\mu \in \mathcal{M}_{f}(M)$, one has
 $$h^{u}_{\mu}(f)+\mathcal{G}_{*}(\mu)\leq P^{u}(f,\mathcal{G}).$$
 

Take any finite partition $\alpha=\{A_{0}, A_{1}, \cdots, A_{k}\}$ of $M$, such that $\mbox{diam}(A_i)\leq \epsilon_{0}$, $ A_{i}$ is compact for $1\leq i\leq k$, and $A_{0}=M\setminus\cup^{k}_{i=1}A_{i}$.  Then for any $n\in \mathbb{N}$, based on the definition of conditional measure and Lemma \ref{lemma 2.9},  one has 
$$
\begin{aligned}
&H_\mu(\alpha_{0}^{n-1}|\eta)+\int_{M}\log g_{n}(x)d\mu(x)\\
&=\int_{M}\log g_{n}(x)-\log\mu^{\eta}_{x}(\alpha_{0}^{n-1}(x))d\mu(x)\\
&=\int_{M}\int_{\eta(x)}\left(\log g_{n}(y)-\log\mu^{\eta}_{x}(\alpha_{0}^{n-1}(y))\right)d\mu^{\eta}_x(y)d\mu(x)\\
&\leq\int_M\left(\sum_{A\in\alpha^{n-1}_0\cap\eta(x)}\mu^{\eta}_x(A)(\sup\limits_{y\in A}\log g_n(y)-\log \mu^{\eta}_x(A))\right)d\mu\\
&\leq\sup\limits_{B\in\eta}\log\left(\sum\limits_{A\in\alpha_{0}^{n-1}\cap B}\sup\limits_{y\in A}g_{n}(y)\right).
\end{aligned}
$$
Since $\text{diam}(\eta) \leq \epsilon_{0}$, then for any $B\in\eta$, there exists $x\in M$ and $\delta>0$ such that $B\subseteq\overline{W^u(x,\delta)}$. Then 
$$H_\mu(\alpha_{0}^{n-1}|\eta)+\int_{M}\log g_{n}(x)d\mu(x)\leq\sup\limits_{x\in M}\log\left(\sum\limits_{A\in\alpha_{0}^{n-1}}\sup\limits_{x\in A\cap\overline{W^u(x,\delta)}}g_{n}(x)\right).$$

Set $b=\min\{d(A_{i}, A_{j}):i,j=1,\cdots,k,\,i\neq j\}$, for any $\epsilon>0$ with $\epsilon<b/2$ and any $n\in\mathbb{N}$, we shall construct an $(n,\epsilon)$ $u$-separated set of $\overline{W^u(x,\delta)}$ to obtain an estimation of $P^u_n(f, \mathcal{G}, \epsilon, x, \delta)$.

Let $ \mathcal{M}=\{ C = A\cap\overline{W^u(x,\delta)}| A\in\alpha_{0}^{n-1}\}$. For each $C\in \mathcal{M}$, choose some $x(C)\in \overline{C}$ such that $g_n\left(x(C)\right)=\sup\limits_{y\in C}g_n(y)$.
We claim that for each $C\in \mathcal{M}$, there are at most $2^{n}$ different $\widetilde{C}$ in $\mathcal{M}$, such that $d^{u}_{n}\left(x(C), x(\widetilde{C})\right)<\epsilon$.
To see this claim, for each $C\in \mathcal{M}$, pick up the unique index tuple $(i_{0}(C), i_{1}(C), . . . , i_{n-1}(C)) \in \{0, 1, 2, . . . , k\}^{n}$
such that
$$C= (A_{i_{0}(C)}\cap f^{-1}A_{i_{1}(C)}\cap f^{-2}A_{i_{2}(C)}\cap \cdots\cap f^{-(n-1)}A_{i_{n-1}(C)})\cap \overline{W^u(x,\delta)}.$$
Now fix a $C \in \mathcal{M}$ and let $\mathcal{Y}$ denote the collection of all $\widetilde{C}$ with $d^{u}_{n}\left(x(C), x(\widetilde{C})\right)<\epsilon,$ then we
have
$$\sharp\{i_{l}(\widetilde{C})|\widetilde{C}\in\mathcal{Y}\}\leq 2,l=0,1,\cdots,n-1.$$
To see this inequality, we assume on the contrary that there exists a $l$, and $\widetilde{C_{1}},\,\widetilde{C_{2}},\,\widetilde{C_{3}}\in\mathcal{Y}$,
 such that $i_{l}( \widetilde{C_{1}}), i_{l}( \widetilde{C_{2}}), i_{l}( \widetilde{C_{3}})$ are distinct. Without loss of generality, we assume $i_{l}( \widetilde{C_{1}})\neq 0,\, i_{l}( \widetilde{C_{2}})\neq0$. This implies
$$
\begin{aligned}
d^{u}_{n}\left(x(\widetilde{C_{1}}), x(\widetilde{C_{2}})\right)&\geq d^{u}_{n}\left(f^{l}(x(\widetilde{C_{1}})), f^{l}(x(\widetilde{C_{2}}))\right)\\
&\geq d^{u}_{n}\left(A_{i_{l}(\widetilde{C_{1}})},A_{i_{l}(\widetilde{C_{2}})}\right)\geq b >2\epsilon,
\end{aligned}
$$ which leads to a contradiction.

Now we choose an element $C_{1} \in \mathcal{M},$ such that $$g_{n}\left(x(C_{1})\right) =\max\limits_{C\in\mathcal{M}}g_{n}\left(x(C)\right).$$
Let $\mathcal{Y}_{1}$ denote the collection of all $\widetilde{C}\in\mathcal{Y}$ with $d^{u}_{n}\left(x(C_{1}), x(\widetilde{ C})\right) < \epsilon.$ Then the cardinality of $\mathcal{Y}_{1}$ does not exceed $2^{n}$. If the collection $\mathcal{M}\setminus\mathcal{Y}_{1}$ is not empty, we choose an element $C_{2} \in \mathcal{M}\setminus\mathcal{Y}_{1}$ such that
$$g_{n}\left(x(C_{2})\right) =\max\limits_{C\in \mathcal{M}\setminus\mathcal{Y}_{1}}g_{n}(x(C)).$$
Let $\mathcal{Y}_{2}$ be the collection of $\widetilde{C}\in \mathcal{M}\setminus\mathcal{Y}_{1}$ with $d^{u}_{n}(x(C_{2}), x(\widetilde{C}))<\epsilon.$ We continue this process, in step $m$ we choose an element $C_{m}\in \mathcal{M}\setminus\cup^{m-1}_{i=1}\mathcal{Y}_{i}$
such that
$$g_{n}(x(C_{m})) =\max\limits_{C\in\mathcal{M}\setminus\cup^{m-1}_{i=1}\mathcal{Y}_{i}}g_{n}(x(C)).$$
Let $\mathcal{Y}_{m}$ be the collection of $\widetilde{C} \in\mathcal{M}\setminus\cup^{m-1}_{i=1}\mathcal{Y}_{i}$ with $d^{u}_{n}(x(C_{m}), x(\widetilde{C}))< \epsilon.$ Since the partition is finite, the process above will stop at some step $m$. Set $E = \{x(C_{j})\mid j =1, 2, . . .,m\}$. Then $E$ is an $(n,\epsilon)$ $u$-separated set of $\overline{W^u(x,\delta)}$.

For each $\mathcal{Y}_{j}$, we have
$$\sum_{C\in \mathcal{Y}_{j}}g_{n}(x(C))\leq 2^{n}g_{n}(x(C_{j})).$$
Then
$$
\begin{aligned}
\sum_{y\in E}g_{n}(y)=\sum^{m}_{j=1}g_{n}(x(C_{j}))&\geq\sum^{m}_{j=1}\frac{1}{2^{n}}\sum_{C\in\mathcal{Y}_{j}}g_{n}(x(C))\\&=\frac{1}{2^{n}}\sum_{A\in\alpha_{0}^{n-1}}\sup\limits_{x\in A\cap\overline{W^u(x,\delta)}}g_{n}(x).
\end{aligned}
$$
Hence
$$\sum_{A\in\alpha_{0}^{n-1}}\sup\limits_{x\in A\cap\overline{W^u(x,\delta)}}g_{n}(x)\leq 2^{n}\sum_{y\in E}g_{n}(y),$$
and
$$
\begin{aligned}
H_\mu(\alpha_{0}^{n-1}|\eta)+\int_{M}\log g_{n}(x)d\mu(x)&\leq\sup\limits_{x\in M}\log\left(\sum\limits_{A\in\alpha_{0}^{n-1}}\sup\limits_{x\in A\cap\overline{W^u(x,\delta)}}g_{n}(x)\right)\\
&\leq\sup\limits_{x\in M}\log(2^{n}\sum_{y\in E}g_{n}(y)).
\end{aligned}
$$
Divided by $n$ on both sides, one has
$$
\begin{aligned}
\frac{1}{n}H_\mu(\alpha_{0}^{n-1}|\eta)+\frac{1}{n}\int_{M}\log g_{n}(x)d\mu(x)&\leq\log2+\frac{1}{n}\sup\limits_{x\in M}\log(\sum_{y\in E}g_{n}(y))\\
&\leq\log2+\frac{1}{n}\log P_{n}^{u}(f,\mathcal{G},\epsilon,\overline{W^u(x,\delta)}).
\end{aligned}
$$
Take the limit superior with $n\rightarrow\infty$, one gets 
$$h_{\mu}(\alpha|\eta)+\mathcal{G}_{*}(\mu)\leq\log 2+\sup\limits_{x\in M}P^{u}(f,\mathcal{G},\epsilon,\overline{W^u(x,\delta)}).$$
Let $\epsilon\rightarrow 0,\,\delta\rightarrow 0$, one has
$$h^{u}_{\mu}(f)+\mathcal{G}_{*}(\mu)\leq\log2+P^{u}(f,\mathcal{G}).$$
Since this inequality holds for all diffeomorphisms and sub-additive potentials, thus we
can apply it to $f^{k}$ and $\mathcal{G}^{(k)}$ and get
$$h^{u}_{\mu}(f^{k})+\mathcal{G}^{(k)}_{*}(\mu)\leq\log2+P^{u}(f^{k},\mathcal{G}^{(k)}).$$
By Proposition \ref{propeties} and Lemma \ref{lemma 2.10}, we have
$$h^{u}_{\mu}(f)+\mathcal{G}_{*}(\mu)\leq\frac{\log2}{k}+P^{u}(f,\mathcal{G}).$$
Since $k$ is arbitrary, we have
$$h^{u}_{\mu}(f)+\mathcal{G}_{*}(\mu)\leq P^{u}(f,\mathcal{G}).$$

$\mathit{Step \,2}.$
In this step we show that for any $\rho>0$, there exists a $\mu\in \mathcal{M}_f(M)$ satisfying $$h^u_\mu(f)+\mathcal{G}_*(\mu)\geq P^u(f, \mathcal{G})-\rho.$$ The argument is quite similar to the case of additive potentials, as well as the classical case of entropy, we exhibit the full detail for the convenience of readers.
By the definition of $P^u(f, \mathcal{G})$, for any $\delta>0$ small enough, one can take $x\in M$ such that
$$P^u\left(f, \mathcal{G}, \overline{W^u(x, \delta)}\right)\geq P^u(f, \mathcal{G})-\rho.$$
For any small $\epsilon>0$, suppose $E_n$ be an $(n, \epsilon)$ $u$-separated set of $\overline{W^u(x, \delta)}$ whose cardinality is denoted by $N^u(f, \epsilon, n, x, \delta)$ satisfying
$$\log\sum_{y\in E_n}g_n(y)\geq \log P^u_n(f, \mathcal{G}, \epsilon, x, \delta)-1.$$
Define
$$\nu_n:=\dfrac{\sum_{y\in E_n}g_n(y)\delta_y}{\sum_{z\in E_n}g_n(z)}\;\text{and}\;\mu_n:=\dfrac{1}{n}\sum_{i=0}^{n-1}\nu_n\circ f^{-i}.$$
By compactness of $\mathcal{M}(M)$ equipped with the weak*-topology, we can find a subsequence $\{n_k\}$ such that $\{\mu_{n_k}\}$ converges, say $\lim\limits_{k \to \infty}\mu_{n_k}=\mu$, obviously $\mu\in\mathcal{M}_f(M)$. Next we show that  $\mu$ fits in our purpose.

For some $\delta$ small enough, pick $\eta\in\mathcal{P}^u$ such that $W^u(x, \delta)\subset\eta(x)$. Then one can choose a $\alpha\in\mathcal{P}$ with $\mu(\partial\alpha)=0$, and $\text{diam}(\alpha)<\epsilon/C$ where $C>1$ is as in (\ref{distance}). Hence $\log N^u(f, \epsilon, n, x, \delta)=H_{\nu_n}(\alpha_{0}^{n-1}|\eta).$ For any given $q>1$,  put $a(j)=[\dfrac{n-j}{q}]$, where $n$ is a natural number with $n>q$ and $j=0,1,...,q-1$.
So
$$\bigvee_{i=0}^{n-1}f^{-i}\alpha=\bigvee_{r=0}^{a(j)-1}f^{-(rq+j)}\alpha_{0}^{q-1}\vee\bigvee_{t\in S_j}f^{-t}\alpha,$$ where $S_j=\{0,1,..., j-1\}\cup\{j+qa(j),..., n-1\}$.

For any $\alpha\in\mathcal{P}$, suppose $\alpha^u$ is the partition in $\mathcal{P}^u$ whose elements are given by $\alpha^u(x)=\alpha(x)\cap W^u_{loc}(x)$. Then
 \begin{align}\label{(13)}
f^{rq}\left(\bigvee_{i=0}^{r-1}f^{-iq}\alpha_{0}^{q-1}\vee f^j\eta\right)
=f^{rq}(\alpha_{0}^{rq-1}\vee f^j\eta)
=f\alpha\vee...\vee f^{rq}\alpha\vee f^{rq+j}\eta\geq f\alpha^u,
\end{align}
and
\begin{equation}\label{(14)}
H_{\nu}\left(\bigvee_{r=0}^{a(j)-1}f^{-(rq+j)}\alpha_{0}^{q-1}|\eta\right)=H_{f^j_*\nu}\left(\bigvee_{r=0}^{a(j)-1}f^{-rq}\alpha^{q-1}_{0}|f^j\eta\right), \forall\, \nu\in\mathcal{M}_f(M).
\end{equation}
On one hand, Lemma \ref{lemma 2.9} implies that
$$\begin{aligned}
H_{\nu_n}(\alpha_{0}^{n-1}|\eta)+\int_{M}\log g_{n}d\nu_n&=\sum_{y\in E_n}\nu_n(\{y\})\left(-\log\nu_n(\{y\})+\log g_n(y)\right)\\&=\log\sum_{y\in E_n}g_n(y).
\end{aligned}
$$ On the other hand, $$\begin{aligned}
&H_{\nu_n}(\alpha_{0}^{n-1}|\eta)+\int_M \log g_n d\nu_n\\
=&H_{\nu_n}\left(\bigvee_{r=0}^{a(j)-1}f^{-(rq+j)}\alpha_{0}^{q-1}\vee\bigvee_{t\in S_j}f^{-t}\alpha|\eta\right)+\int_M \log g_n d\nu_n\\
\leq&\sum_{t\in S_j}H_{\nu_n}(f^{-t}\alpha|\eta)+H_{\nu_n}\left(\bigvee_{r=0}^{a(j)-1}f^{-(rq+j)}\alpha_{0}^{q-1}|\eta\right)+\int_M \log g_n d\nu_n\\
=&\sum_{t\in S_j}H_{\nu_n}(f^{-t}\alpha|\eta)+H_{f^{j}\nu_n}\left(\bigvee_{r=0}^{a(j)-1}f^{-rq}\alpha_{0}^{q-1}|f^{j}\eta\right)+\int_M \log g_n d\nu_n.
\end{aligned}
$$ While by Lemma 2.6 in \cite{Hu1}, combining with (\ref{(13)}) and (\ref{(14)}), one has
$$
\begin{aligned}
&H_{f^j_*\nu_{n}}(\bigvee_{r=0}^{a(j)-1}f^{-rq}\alpha_{0}^{q-1}|f^j\eta)\\
=&H_{f^i_*\nu_n}(\alpha_{0}^{q-1}|f^j\eta)+\sum_{r=1}^{a(j)-1}H_{f^{j}_*\nu_n}\left(\alpha_0^{q-1}|f^{rq}(\bigvee_{i=0}^{r-1}f^{-iq}\alpha_{0}^{q-1}\vee f^j\eta)\right)\\
\leq&H_{f^j_*\nu_n}(\alpha_{0}^{q-1}|f^j\eta)+\sum_{r=1}^{a(j)-1}H_{f^{i}_*\nu_n}(\alpha_0^{q-1}|f\alpha^u).
\end{aligned}
$$
It is clear that $|S_j|\leq 2q$ and assume $|\alpha|=d$. Add up the inequalities above over $j$ from $0$ to $q-1$, and divided by $n$,  one gets 
\begin{align}\label{15}
&\dfrac{q}{n}\log\sum_{y\in E_n}g_n(y)
\leq\dfrac{1}{n}\sum_{j=0}^{q-1}\sum_{t\in S_j}H_{\nu_n}(f^{-t}\alpha\mid\eta)+\dfrac{1}{n}\sum_{j=0}^{q-1}H_{f^j_*\nu_n}(\alpha_{0}^{q-1}\mid f^j\eta)\notag\\
&+\dfrac{1}{n}\sum_{i=0}^{n-1}H_{f^j_*\nu_n}(\alpha_{0}^{q-1}\mid f\alpha^u)+\dfrac{q}{n}\int_M \log g_n d\nu_n\notag\\
\leq&\dfrac{2q^2}{n}\log d+\dfrac{1}{n}\sum_{j=0}^{q-1}H_{f^j_*\nu_n}(\alpha_{0}^{q-1}\mid f^j\eta)+H_{\mu_n}(\alpha_{0}^{q-1}\mid f\alpha^u)+\dfrac{q}{n}\int_M \log  g_n d\nu_n.
\end{align}

Let $\{n_k\}$ be a sequence of natural numbers satisfying
\\$(1) \,\mu_{n_k}\rightarrow \mu$ as $k\rightarrow \infty.$
\\$(2)\,\lim\limits_{k \to \infty}\dfrac{1}{n_k}\log P^u_{n_k}(f, \mathcal{G}, \epsilon, x, \delta)=\limsup\limits_{n \to \infty}\dfrac{1}{n}\log P^u_{n}(f, \mathcal{G}, \epsilon, x, \delta).$
Since $\mu(\partial\alpha)=0$ and $\mu\in\mathcal{M}_{f}(M)$, then for any $q\in\mathbb{N}$, one has $\mu(\partial\alpha_{0}^{q-1})=0$. By upper semi-continuity of the unstable metric entropy, one has that $$\limsup_{k \to \infty}H_{\mu_{n_k}}(\alpha_{0}^{q-1}\mid f\alpha^u)\leq H_{\mu}(\alpha_{0}^{q-1}\mid f\alpha^u).$$
Now we can deduce from (\ref{15}) that $$q\limsup\limits_{n \to \infty}\dfrac{1}{n}\log P^u_{n}(f, \mathcal{G}, \epsilon, x, \delta)\leq H_{\mu}(\alpha^{q-1}_{0}\mid f\alpha^u)+q\mathcal{G}_*(\mu),$$ and so $$\begin{aligned}
P^u\left(f, \mathcal{G}, \overline{W(x,\delta)}\right)&\leq \lim\limits_{q\to\infty}\dfrac{1}{q}H_{\mu}(\alpha^{q-1}_{0}\mid f\alpha^u)+\mathcal{G}_*(\mu)\\
&=h^u_{\mu}(f)+\mathcal{G}_*(\mu).
\end{aligned}
$$ Hence $h^u_\mu(f)+\mathcal{G}_*(\mu)\geq P^u(f, \mathcal{G})-\rho$.

Combining with $\mathit{Step\,1}$, we have proved $$P^u(f, \mathcal{G})=\sup\left\{h^u_\mu(f)+\mathcal{G}_*(\mu)\mid\mu\in\mathcal{M}_f(M)\right\}.$$

$\mathit{Step\,3}.$ Now we prove the second conclusion of Theorem \ref{variational principle}.

Let $\rho>0$ be sufficiently small, by $\mathit{step\,2}$, there exists an invariant measure $\mu$ such that $$h^u_\mu(f)+\mathcal{G}_*(\mu)>P^u(f, \mathcal{G})-\rho.$$ Note that $
h^u_\mu(f)+\mathcal{G}_*(\mu)=\int_{\mathcal{M}^e_f(M)}\left(h^u_\nu(f)+\mathcal{G}_*(\nu)\right)d\nu,$ there is an ergodic invariant measure $\nu$ such that $$h^u_\nu(f)+\mathcal{G}_*(\nu)
>P^u(f, \mathcal{G})-\rho.$$ Let $\rho\to0$, one has $P^u(f, \mathcal{G})=\sup\{h^u_\mu(f)+\mathcal{G}_*(\mu)\mid\mu\in\mathcal{M}^e_f(M)\}.$
\end{proof}

Inspired by \cite{Cao}, we give the following equivalent description of Theorem \ref{variational principle}.

\begin{prop}\label{limits}
Let $f : M \rightarrow M$ be a $C^{1}$-smooth partially hyperbolic diffeomorphism and $\mathcal{G} = \{\log g_{n}\}
_{n=1}^{\infty}$ be a sequence of sub-additive potentials of $f$ on $M$. Then
$$P^u(f, \mathcal{G})=\lim\limits_{n\rightarrow\infty}P^u(f, \frac{\log g_{n}}{n})$$
if and only if
$$
P^{u}(f,\mathcal{G})=\sup\{h^{u}_{\mu}(f)+\mathcal{G}_{*}(\mu)\mid\mu\in \mathcal{M}_{f}(M)\}.
$$
\end{prop}

\begin{proof}
$\textbf{{The ''if'' part}}.$
By assumption, for any $\mu\in\mathcal{M}_{f}(M)$ we have \begin{equation}\label{sequece-limit(2)}
h^{u}_{\mu}(f)+\lim\limits_{n\rightarrow\infty}\int\frac{\log g_{n}}{n}d\mu\leq P^u(f, \mathcal{G}).
\end{equation}
 Based on upper semi-continuity of unstable metric entropy, for any $t\in\mathbb{Z}^{+}$, there exists $\mu_{2^{t}}\in\mathcal{M}_{f}(M)$ such that
$$P^u(f, \frac{\log g_{2^{t}}}{2^{t}})=h^{u}_{\mu_{2^{t}}}(f)+\int\frac{\log g_{2^{t}}}{2^{t}}d\mu_{2^{t}}.$$
Since the set $\mathcal{M}_{f}(M)$ is compact, we can assume that $\{\mu_{2^{t}}\}_{t\in\mathbb{Z}^{+}}\rightarrow\mu$. By the sub-additivity of $\{\log g_{n}\}_{n\in\mathbb{Z}^{+}}$, the sequence $\{\dfrac{\log g_{n}}{n}\}$ is decreasing. Then
$$h^{u}_{\mu_{2^{t}}}(f)+\int\frac{\log g_{2^{t}}}{2^{t}}d\mu_{2^{t}}\leq h^{u}_{\mu_{2^{t}}}(f)+\int\frac{\log g_{2^{s}}}{2^{s}}d\mu_{2^{t}},$$ if $s<t$.
 Based on upper semi-continuity of unstable metric entropy, one has $$\begin{aligned}
&\lim\limits_{n\rightarrow\infty}P^u(f, \frac{\log g_{n}}{n})=\lim\limits_{t\rightarrow\infty}P^u(f, \frac{\log g_{2^{t}}}{2^{t}})\\
=&\lim\limits_{t\rightarrow\infty}\left(h^{u}_{\mu_{2^{t}}}(f)+\int\frac{\log g_{2^{t}}}{2^{t}}d\mu_{2^{t}}\right)\\
\leq &\lim\limits_{t\rightarrow\infty}\left(h^{u}_{\mu_{2^{t}}}(f)+\int\frac{\log g_{2^{s}}}{2^{s}}d\mu_{2^{t}}\right)\;(s<t)\\
\leq&h^{u}_{\mu}(f)+\int\frac{\log g_{2^{s}}}{2^{s}}d\mu.
\end{aligned}$$
By the arbitrariness of the natural number $s$, one has
\begin{equation}\label{sequece-limit(1)}
\lim\limits_{n\rightarrow\infty}P^u(f, \frac{\log g_{n}}{n})\leq h^{u}_{\mu}(f)+\lim\limits_{s\rightarrow\infty}\int\frac{\log g_{2^{s}}}{2^{s}}d\mu.
\end{equation}
Combining (\ref{sequece-limit(2)}) with (\ref{sequece-limit(1)}), we get
$$\lim\limits_{n\rightarrow\infty}P^u(f, \frac{\log g_{n}}{n})\leq P^u(f, \mathcal{G}). $$

For the other inequality, we write $n=sl+r$, for any given $l$, where $s\geq0,\,0\leq r< l$, by the sub-additivity of $\{\log g_{n}\}$, for any $0\leq j< l$, we have
$$\log g_{n}(x)\leq \log g_{j}(x)+\log g_{l}(f^{j}x)+\cdots+\log g_{l}(f^{(s-2)l+j}x)+\log g_{l+r-j}(f^{(s-1)l+j}x),$$
where $\log g_{0}(x)\equiv 0$. Summing up the inequalities above from $j=0$ to $j=l-1$ leads to
$$l\log g_{n}(x)\leq 2lC_{1}+\sum^{(s-1)l-1}_{i=0}\log g_{l}(f^{i}x),$$
where $C_{1}=\max\limits_{0\leq j\leq 2l-1}\max\limits_{x\in M}|\log g_{j}(x)|$.
Hence
\begin{equation}\label{(2.2)}
\log g_{n}(x)\leq 2C_{1}+\sum^{(s-1)l-1}_{i=0}\dfrac{1}{l}\log g_{l}(f^{i}x)\leq 4C_{1}+\sum^{n-1}_{i=0}\dfrac{1}{l}\log g_{l}(f^{i}x),
\end{equation}
and so
$$g_{n}(x)\leq  \exp(4C_{1})\cdot\exp(\sum^{n-1}_{i=0}\dfrac{1}{l}\log g_{l}(f^{i}x)).$$
Then for any $x\in M$ and any $\delta>0$, one has
$$P_{n}^{u}(f, \mathcal{G}, \epsilon, x, \delta)\leq P_{n}^{u}(f, \dfrac{1}{l}\log g_{l}, \epsilon, x, \delta),$$
$$P^{u}(f, \mathcal{G})\leq P^{u}(f, \dfrac{1}{l}\log g_{l}),$$
and then
$$P^{u}(f, \mathcal{G})\leq\lim\limits_{l\rightarrow\infty}P^{u}(f, \dfrac{1}{l}\log g_{l}).$$

$\textbf{The ''only if " part}.$
Since the entropy map $\mu\rightarrow h^{u}_{\mu}(f)$ is upper semi-continuous, with a similar argument as Proposition $4.4$ in \cite{Huang1}, we have
$$\lim\limits_{n\rightarrow\infty}\sup\{h^{u}_{\mu}(f)+\int\frac{\log g_{n}}{n}d\mu\mid\mu\in \mathcal{M}_{f}(M)\}=\sup\{h^{u}_{\mu}(f)+\mathcal{G}_{*}(\mu)\mid\mu\in \mathcal{M}_{f}(M)\}.$$
Then by the variational principle for the additive potentials, we have $$P^u(f, \frac{\log g_{n}}{n})=\sup\{h^{u}_{\mu}(f)+\int\frac{\log g_{n}}{n}d\mu\mid\mu\in \mathcal{M}_{f}(M)\},$$
and so $$
\begin{aligned}
P^u(f, \mathcal{G})&=\lim\limits_{n\rightarrow\infty}\sup\{h^{u}_{\mu}(f)+\int\frac{\log g_{n}}{n}d\mu:\;\mu\in \mathcal{M}_{f}(M)\}\\
&=\sup\{h^{u}_{\mu}(f)+\mathcal{G}_{*}(\mu)\mid\mu\in \mathcal{M}_{f}(M)\}.
\end{aligned}$$
\end{proof}
\begin{rem}
This proposition shows that to deduce the variational principle for sub-additive potentials from additive case is equivalent to prove it directly. 
\end{rem}

\proof[Acknowledgements] The first author is supported by a NSFC (National Science Foundation of China) grant with grant No.\,11501066 and a grant from the Department of Education in Chongqing City with contract No.\,KJQN201900724 in Chongqing Jiaotong University.

The second author is supported by the Chongqing Key Laboratory of Analytic Mathematics and Applications.

The third author is supported by the National Science Foundation
of China with grant No.\,11871120 and 11671093.

\end{document}